\documentclass[12pt]{article}
\usepackage{amssymb,amsmath,amsthm}
\usepackage{pstricks,pst-node,pst-plot}
\usepackage{comment}
\usepackage{color}
\usepackage{verbatim}
\usepackage{enumerate}
\usepackage[all]{xy}

\def \R {\mathbb R}
\def \C {\mathbb C}

\def \bj {\bar{j}}
\def \bl {\bar{l}}
\def \bm {\bar{m}}
\def \bn {\bar{n}}
\def \bk {\bar{k}}

\def \bi {\bar{i}}
\def \bq {\bar{q}}

\def \db {\bar{\partial}}
\def \d {\partial}
\def \Tr {\text{tr}}

\def \lapl {\Delta}

\def\XXint#1#2#3{{\setbox0=\hbox{$#1{#2#3}{\int}$ }
\vcenter{\hbox{$#2#3$ }}\kern-.6\wd0}}

\newtheorem {Theorem} {Theorem}[section]
\newtheorem {Proposition}[Theorem]      {Proposition}
\newtheorem {Lemma}      [Theorem]       {Lemma}
\newtheorem {Corollary} [Theorem] {Corollary}

\theoremstyle{definition}
\newtheorem  {Definition}{Definition}[section]

\theoremstyle{remark}
\newtheorem{Remark}[Theorem]	{Remark}
\newtheorem{Example}[Theorem]	{Example}

\title{A new positivity condition for the curvature of Hermitian manifolds}
\author{Freid Tong}
\date{}
\newcommand{\Addresses}{{
		\bigskip
		\footnotesize
		\textsc{Department of Mathematics, Columbia University,
			2990 Broadway, New York, NY 10027}\par\nopagebreak
		\textit{E-mail address}: \texttt{tong@math.columbia.edu}
}}

\begin{document}
\maketitle

\let\thefootnote\relax\footnote{This work is supported in part by NSF grant DMS-1855947.}
\begin{abstract}
	In this note, we introduce a new type of positivity condition for the curvature of a Hermitian manifold, which generalizes the notion of nonnegative quadratic orthogonal bisectional curvature to the non-K\"ahler case.  We derive a Bochner formula for closed $(1, 1)$-forms from which this condition appears naturally and prove that if a Hermitian manifold satisfy our positivity condition, then any class $\alpha \in H^{1, 1}_{BC}(X)$ can be represented by a closed $(1, 1)$-form which is parallel with respect to the Bismut connection. Lastly, we show that such a curvature positivity condition holds on certain generalized Hopf manifolds and on certain Vaisman manifolds. 
\end{abstract}
\section{Introduction}
A general principle in K\"ahler geometry has been that positive curvature conditions impose important geometric and topological constraints on the manifold. A manifestation of this is the solution of the Frankel conjecture by Siu-Yau and Mori \cite{SiuYau, Mori}, and its generalization by Mok \cite{Mok}, which can be viewed as a uniformization theorem for K\"ahler manifolds with nonnegative bisectional curvature. In recent years, the study of non-K\"ahler Hermitian geometry has received a lot of attention, partly due to its connection with heterotic string theory \cite{FHP, FeiYau, FuYau, Phong, PPZ1}. It is then natural to look for curvature positivity conditions in non-K\"ahler geometry which can lead to significant geometric and topological consequences. In one such direction, Ustinovskiy \cite{Ust} proposes to study the uniformization of Hermitian manifolds with nonnegative Griffiths curvature by using a geometric flow which preserves the positivity of the Griffiths curvature. As shown by Fei and Phong \cite{FeiPh}, the Hermitian curvature flow which he identified as positive curvature preserving is precisely the same as the flows motivated by string theory introduced in \cite{PPZ1}.

In this article, we introduce a new type of curvature positivity condition involving a tensor $Q$, which is different from the Chern curvature. This tensor $Q$ is less positive than the Chern curvature, hence the corresponding positivity conditions formulated using $Q$ are generally stronger than the corresponding positivity of the Chern curvature. We prove a general Bochner-Kodaira type formula, in which this tensor appears naturally and we show that under a certain $Q$-nonnegativity condition, every class in $\alpha \in H^{1, 1}_{BC}(X)$ can be represented by a closed $(1, 1)$-form which is parallel with respect to the Bismut connection. This can be seen as a generalization of a theorem by Howard-Smyth-Wu \cite{HSW} to the non-K\"ahler setting. Perhaps surprisingly, $Q$ also arises in several different contexts in non-K\"ahler geometry: on one hand, it can be interpreted as a combination of the form $\d\db\omega$ with the $(1,1)$-part of the curvature of the Bismut connection. On the other hand, it acquires an additional symmetry on Vaisman manifolds. We discuss some corresponding applications, which suggest that $Q$ may be very useful for future investigations in non-K\"ahler geometry.

\subsection{Background and conventions}
Let $(X, J, g)$ be a compact Hermitian manifold, then in local holomorphic coordinates, we can write $\omega = \sum_{i, j}g_{i\bj}\sqrt{-1}dz^i\wedge d\bar{z}^j$. 

On the holomorphic tangent bundle, there is a natural connection called the Chern connection, which in holomorphic coordinate is
\begin{align}
	\nabla_iX^j& = \d_iX^j +\Gamma_{il}^jX^l\\
	\nabla_{\bi}X^j &= \d_{\bi}X^j
\end{align}
where the connection coefficients are given by
\[\Gamma_{ki}^j = g^{j\bm}\d_kg_{\bm i}\]
The Torsion tensor of the Chern connection is $T(X, Y)= \nabla_XY-\nabla_YX-[X, Y]$ and it satisfies $T(JX, Y) = T(X, JY)$. Hence the torsion has no $(1, 1)$-component, and it is entirely determined by its $(2, 0)$-component. In local holomorphic coordinates, this is given by
\[ T_{ij}^k = \Gamma_{ij}^k-\Gamma_{ji}^k\]
which is zero iff the metric is K\"ahler. 

The curvature of the Chern connection satisfies the symmetry 
\[R(X, Y, W, Z) = R(JX, JY, W, Z) = R(X, Y, JW, JZ)\]
this implies the curvature has no $(2, 0)$ or $(0, 2)$ components and the entire curvature tensor is determined by the $(1, 1)$-part, which in holomorphic coordinates is given by
\[R_{i\bj k}^{\;\;\;\; l} = -\d_{\bj}\Gamma_{ik}^l.\]
We remark that the Chern connection can be characterized as the unique connection such that is Hermitian i.e $\nabla g =0$ and satisfy $\nabla^{0, 1}  = \db_{T^{1, 0}M}$ where $\nabla^{0, 1}$ is the $(0, 1)$ part of the connection and $\db_{T^{1, 0}M}$ is the d-bar operator on holomorphic vector bundles. 

\subsection{Positivity condition}
For this paper, we are interested in a 4-tensor $Q$, which is defined by the following
\begin{equation}
	Q_{i\bj k\bl} = R_{i\bj k\bl} - g^{p\bq}T_{k p\bj}\overline{T_{l q\bi}}
\end{equation}
It turns out that this tensor naturally appears in a Bochner formula for closed $(1, 1)$-forms and has a nice connection with the Bismut connection. 
\begin{Remark}
It is also worth noting that the contraction of $Q$ in the first two coordinates coincides with the evolution term in the pluriclosed flow defined by Streets and Tian in \cite{ST1}, the same term also appears in the Anomaly flow in dimension three which is studied by Phong, Picard, Zhang in \cite{PPZ1}.
\end{Remark}
We now state the positivity condition that we are interested in, this new positivity condition is expressed in terms of the $Q$ tensor. 
\begin{Definition}
	We say $(X, J, g)$ is {\em $Q$-nonnegative} if in any orthonormal frame (i.e where $g_{i\bj} = \delta_{ij}$) and for any list of real numbers $\lambda_1, \ldots, \lambda_n$, we have
	\begin{equation}\label{eq: Q-nonnegative}
	\sum_{k, m=1}^nQ_{m\bm k\bk}(\lambda_k^2-\lambda_k\lambda_m)\geq 0
	\end{equation}
\end{Definition}

\begin{Remark}
	If $Q$ satisfies an additional symmetry $Q_{i\bj k\bl} = Q_{k\bl i\bj}$, then this expression simplifies
	\begin{equation}
	2\sum_{k, m=1}^nQ_{m\bm k\bk}(\lambda_k^2-\lambda_k\lambda_m) = \sum_{k, m =1}^nQ_{m\bm k\bk}(\lambda_k-\lambda_m)^2
	\end{equation}
	This is the case when the metric is K\"ahler, in that case $Q = R$ and the condition reduces to the nonnegative quadratic orthogonal bisectional curvature condition. (see \cite{ChauTam})  
\end{Remark}

\begin{Remark}
	A simple observation is that if a product manifold $M\times N$ is $Q$-nonnegative, then each of its components must be $Q$-nonnegative.
\end{Remark}

\begin{Example}
	All Bismut-flat manifold are $Q$-nonnegative. Indeed, In it is shown in \cite{ZhaoZheng} that these manifold have to be pluriclosed, and from the discussion in section~\ref{sec: Q and Bismut}, it follows that on these manifolds, $Q$ is equal to the curvature of the Bismut connection which is $0$, hence Bismut-flat manifolds are all $Q$-nonnegative. These manifolds has been classified in \cite{WYZ}, and their universal covers must a product of a compact semisimple Lie group with a real vector space. 
\end{Example}
We will give several other examples where the positivity condition holds in Section 3. 

We now state our main theorem regarding manifolds with $Q$-nonnegative $Q$ tensors. Our theorem implies that on a $Q$-nonnegative manifold, one can always solve the Poincare-Lelong equation for a closed $(1, 1)$-form up to the addition of a Bismut parallel $(1, 1)$-form. 

\begin{Theorem}\label{Main Theorem}
	Suppose $(X, J, g)$ is a compact Hermitian manifold which is $Q$-nonnegative, then any closed $(1, 1)$-form with constant trace is parallel with respect to the Bismut connection. In particular, any class in the Bott-Chern cohomology $H^{1, 1}_{BC}(X)$ contains a representative which is parallel respect to the Bismut connection. 
\end{Theorem}

Our theorem can be viewed as a non-K\"ahler generalization of a theorem of Howard-Smyth-Wu \cite{HSW}, which says for a compact K\"ahler manifold with non-negative quadratic orthogonal bisectional curvature, any class in $H^{1, 1}(X)$ contains a parallel representative given by the unique harmonic form in that class. This turned out to be a key ingredient of the resolution of the generalized Frankel conjecture by Mok \cite{Mok}. The key to the proof of the theorem is a Bochner-Kodaira type identity for closed $(1, 1)$-form, which generalizes the identity in \cite{HSW}. In the K\"ahler case, such an identity was also used by Mok-Siu-Yau \cite{MSY} to study the solution Poincare-Lelong equation in relation to the uniformization conjecture for non-compact K\"ahler manifolds. We refer the readers to \cite{NiTam1, NiTam2} and the references therein for subsequent developments in that direction. 

\section{Q and the Bismut connection}\label{sec: Q and Bismut}
On a Hermitian manifold, there is another canonical connection called the Bismut connection, this is also sometimes called the Strominger connection or the Strominger-Bismut connection in literature. This connection was first written down by physicists (see \cite{Strom, HuWi}) and Bismut rediscovered and used it in \cite{Bis} to prove an index formula for non-K\"ahler manifolds. This connection has received a lot of attention recently because of its relation to the study of non-K\"ahler Calabi-Yau manifolds, and the Hull-Strominger system arising from physics, see \cite{FeiYau} and \cite{PPZ2} where the Bismut connection is used in an crucial way. It is also a natural connection to in the study of pluriclosed metrics via the geometric flow as defined by Streets and Tian \cite{ST1, ST2}. 

\begin{Definition}[Bismut connection]
	The Bismut connection is the connection on $T^{1, 0}M$ which in local holomorphic coordinates is given by
	\begin{align}
	\nabla^+_iX^j& = \d_iX^j +\Gamma_{li}^jX^l\\
	\nabla^+_{\bi}X^j &= \d_{\bi}X^j+g^{j\bm}\overline{T_{im\bl}}X^l
	\end{align}
	where $\Gamma_{ki}^j = g^{j\bm}\d_kg_{\bm i}$ and $T_{im\bl} = \d_ig_{m\bl}-\d_mg_{i\bl}$. 
\end{Definition}

\begin{Remark}
	One can check that this connection respects both the metric and the complex structure (i.e $\nabla^+ g = \nabla^+J = 0$) and the corresponding Torsion tensor $T^+(X, Y, Z) = g(\nabla^+_XY-\nabla^+_YX-[X, Y], Z)$ is skew-symmetric. These properties uniquely characterizes the Bismut connection. 
\end{Remark}

The curvature of the Bismut connection does not satisfy the same symmetries as the curvature of the Chern connection. Thus in general, the curvature has a decomposition into $(2, 0), (1, 1)$ and $(0, 2)$-parts, and we compute the $(1, 1)$-part of the curvature of the Bismut connection in terms of the curvature and torsion of the Chern connection. 
\begin{Lemma}
	The $(1, 1)$-part of the curvature of the Bismut connection is given by 
	\begin{equation}
	B_{i\bj k\bl}  = R_{k\bj i\bar{l}}-R_{i\bj k\bar{l}}+R_{i\bar{l}k\bj}-T_{ki\bar{\gamma}}\overline{T_{jl}^{\gamma}}-g^{\gamma\bar{\kappa}}T_{i\gamma \bar{l}}\overline{T_{j\kappa \bar{k}}}
	\end{equation}
	where $B$ is the curvature of the Bismut connection and $R$ is the curvature of the Chern connection. 
\end{Lemma}
\begin{proof}
	The connection coefficients of the Bismut connection is given by
	\[\mathcal{A}_{ik}^l = \Gamma_{ki}^l \]
	\[\mathcal{A}_{\bj k}^l = g^{l\bq}\overline{T_{jq\bar{k}}}\]
	and the $(1, 1)$-part of the curvature is given by
	\begin{align}
	B_{i\bj m}^{\;\;\; \;k} &= \d_i\mathcal{A}_{\bj m}^k-\d_{\bj}\mathcal{A}_{i m}^k+\mathcal{A}_{i l}^k\mathcal{A}_{\bj m}^l-\mathcal{A}_{\bj l}^k\mathcal{A}_{i m}^l\\
	& = g^{k\bl}\nabla_i\overline{T_{j l \bm}}+R_{m\bj i}^{\;\;\;\;\;k}-T_{i l}^k\mathcal{A}_{\bj m}^l+\mathcal{A}_{\bj l}^kT_{i m}^l\\
	& = g^{k\bl}R_{i \bl m\bj}-R_{i \bj m}^{\;\;\;\; k}+R_{m\bj i}^{\;\;\;\;\;k}-g^{l\bn}T_{i l}^k\overline{T_{j n\bm}}+g^{k\bn}\overline{T_{jn\bl}}T_{i m}^l\\
	\end{align}
	lowering the last index gives the result. 
\end{proof}

\begin{Proposition}\label{prop: Q and Strom-Bismut conn}
	The following formula holds
	\begin{equation}
	Q_{i\bj k\bl} = B_{k\bl i \bj}+\d_{\bl}T_{ik\bj}-\d_{\bj}T_{ik\bl}
	\end{equation}
	where $B_{k\bl i\bj}$ is the $(1, 1)$ part of the curvature of the Bismut connection. 
\end{Proposition}

\begin{proof}[proof of Proposition~\ref{prop: Q and Strom-Bismut conn}]
	In local coordinates where $g_{i\bj} = \delta_{ij}$ and $\d_ig_{j\bl} = \frac{1}{2}T_{ij\bl}$, we compute
	\begin{align}
	\d_{\bar{l}}T_{ki\bj}-\d_{\bj}T_{ki\bar{l}} &= \d_{\bar{l}}\d_kg_{i\bj}-\d_{\bar{l}}\d_{i}g_{k\bj}-\d_{k}\d_{\bj}g_{i\bar{l}}+\d_{\bj}\d_{i}g_{k\bar{l}}\\
	& = R_{i\bar{l}k\bj}+R_{k\bj i\bar{l}}-R_{i\bj k\bar{l}}-R_{k\bar{l}i\bj}-\sum_{\gamma}T_{ki\bar{\gamma}}\overline{T_{jl\bar{\gamma}}}\\
	& = B_{i\bj k\bar{l}}-Q_{k\bar{l}i\bj}
	\end{align}
\end{proof}

\begin{Remark}
	The extra term $\d_{\bl}T_{ik\bj}-\d_{\bj}T_{ik\bl}$ is the components of the tensor $\d\db\omega$, and the previous lemma implies that if $\omega$ is pluriclosed, then $Q$ is simply the $(1, 1)$ part of the curvature of the Bismut connection. 
\end{Remark}

\section{Q and a Bochner-Kodaira identity}\label{sec: Bochner-Kodaira}
In this section, we prove Theorem~\ref{Main Theorem}. The key is a Bochner-Kodaira type formula for $(1, 1)$-forms, from which the tensor $Q$ naturally emerges. 

In the K\"ahler case, this formula is well-known and is intimately related to the uniformization of manifolds with nonnegative bisectional curvature. It first appeared in \cite{BG}, and is used in \cite{HSW} to deduce a splitting theorem for manifolds with nonnegative bisectional curvature, which was one of the key imput in the solution of the generalized Frankel conjecture by Mok \cite{Mok}. 

\begin{Theorem}
	The following formula holds for $\rho$ a real closed $(1, 1)$-form with $g^{i\bj}\rho_{i\bj}= const$, 
	\begin{equation}\label{eq; bochner-formula}
	g^{i\bj}\d_i\d_{\bj} |\rho|^2 =2|\nabla^+\rho|^2 +2Q_{i\bj k\bl}(g^{i\bj}(\rho^2)^{k\bl}-\rho^{i\bj}\rho^{k\bl})
	\end{equation}
	where $Q_{i\bj k\bl} = R_{i\bj k \bl}-g^{p\bq}T_{kp\bj}\overline{T_{lq\bi}}$. 
\end{Theorem}

\begin{proof}
	In local holomorphic coordinates, we have $\rho = i\rho_{k\bj}\,dz^{k}\wedge d\bar{z}^j$ where $\overline{\rho_{k\bj}} = \rho_{j\bk}$. By the closedness of $\rho$, we have
	\begin{equation}
	\nabla_{i}\rho_{k\bj}= \nabla_{k}\rho_{i\bj} + T_{ki}^m\rho_{m\bj}
	\end{equation}
	and
	\begin{equation}
	\nabla_{\bj}\rho_{i\bl} = \nabla_{\bl}\rho_{i\bj} + \overline{T_{lj}^n}\rho_{i\bn}
	\end{equation}
	we compute 
	\begin{align}
	\nabla_i\nabla_{\bj}\rho_{k\bl} &= \nabla_i\nabla_{\bl}\rho_{k\bj}+\nabla_i[\overline{T_{lj}^m}\rho_{k\bm}]\\
	&= \nabla_{\bl}\nabla_i\rho_{k\bj} + [\nabla_i, \nabla_{\bl}]\rho_{k\bj}+\overline{T^{m}_{lj}}\nabla_i\rho_{k\bm}+[\overline{R_{j\bi l}^{\;\;\; m}}-\overline{R_{l\bi j}^{\;\;\; m}}]\rho_{k\bm}\\
	& = \nabla_{\bl}\nabla_{k}\rho_{i\bj}+\nabla_{\bl}[T_{ki}^m\rho_{m\bj}] + [\nabla_i, \nabla_{\bl}]\rho_{k\bj}+\overline{T^{m}_{lj}}\nabla_i\rho_{k\bm}+[\overline{R_{j\bi l}^{\;\;\; m}}-\overline{R_{l\bi j}^{\;\;\; m}}]\rho_{k\bm}\\
	& = \nabla_{\bl}\nabla_{k}\rho_{i\bj}+T_{ki}^m\nabla_{\bl}\rho_{m\bj}+\overline{T_{lj}^m}\nabla_i\rho_{k\bm}+[R_{i\bl k}^{\;\;\; m}-R_{k\bl i}^{\;\;\; m}]\rho_{m\bj}+[\overline{R_{j\bi l}^{\;\;\; m}}-\overline{R_{l\bi j}^{\;\;\; m}}]\rho_{k\bm}\\
	&\qquad - R_{i\bl k}^{\;\;\; m}\rho_{m\bj} + R_{i\bl\;\;\bj}^{\;\;\bn}\rho_{k\bn}\\
	& = \nabla_{\bl}\nabla_{k}\rho_{i\bj}+T_{ki}^m\nabla_{\bl}\rho_{m\bj}+\overline{T_{lj}^m}\nabla_i\rho_{k\bm} -R_{k\bl i}^{\;\;\; m}\rho_{m\bj}+R_{i\bj\;\;\bl}^{\;\;\bm}\rho_{k\bm} 
	\end{align}
	differentiating $|\rho|^2$ once, we get
	\begin{equation}
	\d_{\bj}|\rho|^2 =\nabla_{\bj} (g^{p\bq}g^{k\bl}\rho_{k\bq}\rho_{p\bl})= g^{p\bq}g^{k\bl}(\nabla_{\bj}\rho_{k\bq}\rho_{p\bl}+\nabla_{\bj}\rho_{p\bl}\rho_{k\bq})  = 2g^{p\bq}g^{k\bl}\nabla_{\bj}\rho_{k\bq}\rho_{p\bl}
	\end{equation}
	taking the divergence and using our previous calculation for $\nabla_i\nabla_{\bj}\rho_{k\bq}$ gives
	\begin{align}
	\lapl |\rho|^2 &= 2|\nabla \rho|^2+2g^{p\bq}g^{k\bl}g^{i\bj}\nabla_i\nabla_{\bj}\rho_{k\bq}\rho_{p\bl}\\
	& = 2|\nabla \rho|^2+2\langle\d\db \Tr \rho, \rho\rangle+{2g^{p\bq}g^{k\bl}[g^{i\bj}(T_{ki}^m\nabla_{\bq}\rho_{m\bj}+\overline{T_{qj}^m}\nabla_i\rho_{k\bm})]\rho_{p\bl}}\\
	& \qquad + 2g^{k\bl}\Tr_{12}(R)^{\bm p}\rho_{k\bm}\rho_{p\bl}-2R^{\bl p\bj m}\rho_{m\bj}\rho_{p\bl}
	\end{align}
	
	We now specialize to a coordinate system where $g_{i\bj} = \delta_{i\bj}$ and $\rho_{i\bj} = \lambda_i\delta_{i\bj}$, and use the assumption $g^{i\bj}\rho_{i\bj} = const$, 
	\begin{equation}
	\lapl |\rho|^2 = 2|\nabla_i \rho|^2+ \sum\left(\cdots\right) + 2\sum_{k, m=1}^nR_{m\bm k\bk}(\lambda_k^2-\lambda_k\lambda_m )
	\end{equation}
	where 
	\begin{align}
	\sum\left(\cdots\right) &= 2\sum_{k, m, i=1}^n [T_{ki\bm}\nabla_{\bk}\rho_{m\bi}+\overline{T_{ki\bm}}\nabla_i\rho_{k\bm}]\lambda_k\\
	& = 2\sum_{i, k, m=1}^n[T_{ki\bm}\nabla_{\bk}\rho_{m\bi}+\overline{T_{ki\bm}}\nabla_k\rho_{i\bm}]\lambda_k+2\sum_{i, k, m=1}^n\overline{T_{ki\bm}}T_{ki\bm}\lambda_m\lambda_{k}
	\end{align}
	we can complete the square for the gradient term and we get
	\begin{align}
	\lapl|\rho|^2& = 2\sum_{i, k, m = 1}^n|\nabla_{k}\rho_{i\bm}+T_{ki\bm}\lambda_{k}|^2 +2\sum_{i, k, m=1}^n|T_{ki\bm}|^2\lambda_m\lambda_{k}\\
	&\qquad -2\sum_{k,i,m=1}^n|T_{ki\bm}|^2\lambda_k^2+2\sum_{k, m =1}^n(R_{m\bm k\bk}(\lambda_k^2-\lambda_k\lambda_m))\\
	& =  2\sum_{i, k, m = 1}^n|\nabla_{k}\rho_{i\bm}+T_{ki\bm}\lambda_{k}|^2 - 2\sum_{k, m=1}^n\left(R_{ m\bm k\bk}-\sum_{i=1}^nT_{ki\bm}\overline{T_{ki\bm}}\right)\lambda_k\lambda_m\\
	& \qquad + 2\sum_{k, m=1}^n\left(R_{ m\bm k\bk}-\sum_{i=1}^nT_{ki\bm}\overline{T_{ki\bm}}\right)\lambda_k^2\\
	&  =2\sum_{i, k, m = 1}^n|\d_{i}\rho_{k\bm}-\Gamma_{ki\bm}\lambda_m-T_{ik\bm}\lambda_{k}|^2+2\sum_{k, m=1}^n\left(R_{ m\bm k\bk}-\sum_{i=1}^nT_{ki\bm}\overline{T_{ki\bm}}\right)(\lambda_k^2-\lambda_k\lambda_m)\\
	& = 	2|\nabla^+\rho|^2+2\sum_{k, m=1}^nQ_{m\bm k\bk}(\lambda_k^2-\lambda_k\lambda_m)
	\end{align}
	where $Q_{i\bj k\bl} = R_{i\bj k \bl}-g^{p\bq}T_{kp\bj}\overline{T_{lq\bi}}$. 
\end{proof}

\begin{Corollary}\label{Main Corollary}
	Suppose $(X, J, g)$ is $Q$-nonnegative, then any closed real $(1, 1)$-form $\rho$ with constant trace is parallel with respect to the Bismut connection. Moreover, $T(X, Y, \bar{Z}) = 0$ if $X, Y, \bar{Z}$ are in different eigenspaces of $\rho$ corresponding to eigenvalues $\lambda_X, \lambda_Y, \lambda_Z$ and $\lambda_X+\lambda_Y-\lambda_Z \neq 0$.
\end{Corollary}
\begin{proof}
	If $(M, J, g)$ is $Q$-nonegative, then we have
	\[\lapl |\rho|^2 \geq 2|\nabla^{+}\rho|^2\geq 0\]
	and by the maximum principle, we must have $|\rho|^2 = const$ which implies $\lapl |\rho|^2 =0$ so the right hand side of equation~\ref{eq; bochner-formula} must be identically 0, hence we have $\nabla^+\rho = 0$. If we work in coordinates where $g_{i\bj} = \delta_{ij}$ and $\rho_{i\bj} =\lambda_i\delta_{ij}$, then $\nabla^+\rho = 0$ and $\rho$ being closed implies that
	\begin{equation}
		T_{ki\bm}(\lambda_k+\lambda_i-\lambda_m) = 0
	\end{equation}
	from which the second part of the proposition follows. 
\end{proof}

\begin{proof}[Proof of Theorem \ref{Main Theorem}]
	The first statement follows by Corollary~\ref{Main Corollary}, we will prove the second statement that every class in $H^{1, 1}_{BC}(X)$ contains a Bismut parallel representative. Recall that 
	\[H^{1, 1}_{BC}(X) : = \frac{\{\alpha \in \Omega^{1, 1}| d\alpha = 0\}}{i\d\db C^{\infty}(X, \C)},\]
	hence this statement amounts to showing that for any closed $(1, 1)$-form $\alpha$, there exist a smooth function $u$ such that $\nabla^+(\alpha+i\d\db u) = 0$. Let $\alpha$ be a real closed $(1, 1)$-form. In \cite{Gau}, Gauduchon proved that every Hermitian metric is conformally equivalent to a Gauduchon metric, so let $\hat{\omega} =e^f\omega$ be a Gauduchon metric in the conformal class of $\omega$. By the Gauduchon property, we know the equation $i\d\db u\wedge\hat{\omega}^{n-1}= f\hat{\omega}^n$ has a solution iff $\int_X f \hat{\omega}^n = 0$. Let $u$ solve the equation
	\begin{equation}
		 n\frac{\sqrt{-1}\d\db u \wedge \hat{\omega}^{n-1}}{\hat{\omega}^n} = \Tr_{\hat{\omega}}\alpha+ce^{-f}
	\end{equation}
	where $c$ is the constant given by 
	\[c = -\frac{\int_X (\Tr_{\hat{\omega}}\alpha) \hat{\omega}^n}{\int_X e^{-f}\hat{\omega}^n}\]
	Setting $\rho = \alpha-\sqrt{-1}\d\db u $, then $d\rho = 0$ and $\Tr_{\omega}\rho = \Tr_{\omega}\alpha-\lapl_{\omega} u = c$, hence by Corollary \ref{Main Corollary}, we have $\nabla^+\rho = 0$. If $\alpha$ is not real then we can write $\alpha = u+iv$ where $u, v$ are the real and imaginary parts of $\alpha$, then $u, v$ are both real closed $(1, 1)$ forms and we can apply the above argument to both $u$ and $v$. 
\end{proof}

\begin{Corollary}
	If  $(X, J, g)$ is non-K\"ahler, $Q$-nonnegative and $H^{1, 1}_{BC}\neq 0$. Then the holonomy of $\nabla^+$ is contained in a subspace $U(m)\times U(n-m)\subset U(n)$. 
\end{Corollary}

\begin{proof}
	Since $H^{1, 1}_{BC}\neq 0$, by Theorem~\ref{Main Theorem}, there exist a class $0\neq [\rho]\in H^{1, 1}_{BC}$ containing a Bismut flat representative $\rho$, and furthermore, since $X$ is non-K\"ahler, $\rho$ is not a multiple of $g$. Hence if $\lambda_1, \ldots, \lambda_n$ are the eigenvalues of $\rho$ with respect to $g$, then there exist at least two distinct eigenvalues, and since both $g$ and $\rho$ is flat with respect to the Bismut connection, it follows that the eigenspaces of $\rho$ are invariant under the holonomy $\nabla^+$, and we have our result. 
\end{proof}

\section{Q and Vaisman manifolds}\label{sec: Vaisman manifolds}
In this section, we study this positivity condition on a class of Locally conformally K\"ahler manifolds called Vaisman manifolds. These manifolds were first introduced by Vaisman \cite{Vaisman} as an important class of complex non-K\"ahler manifold which he called generalized Hopf manifolds. By \cite{OV: Vaisman, OV: LCK-Rank}, a large class of these manifolds can be viewed as a Sasakian manifolds equipped with a Sasakian automorphism. For more on these manifolds, we refer the reader to \cite{OV: Vaisman, OV: LCK-Rank} and the references therein. 

\begin{Definition}\label{def: LCK/Vaisman manifolds}
	\begin{enumerate}
		\item[(1)] A compact Hermitian manifold $(M, J, g)$ is called {\em locally conformally K\"ahler} if there exist an closed form $\theta$ such that
	\begin{equation}
		d\omega = \theta\wedge\omega
	\end{equation}
	The 1-form $\theta$ is called the {\em Lee form}. 
	\item[(2)] A locally conformally K\"ahler manifold is a {\em Vaisman manifold} if  $\nabla^{LC} \theta = 0$, where $\nabla^{LC}$ is the Levi-Civita connection of $(M, J, g)$. 
	\end{enumerate}
\end{Definition}
\begin{Remark}
A locally conformally K\"ahler metric can always locally be written as a K\"ahler metric times a conformal factor, and the coverse is clearly true as well. Indeed, since the Lee form $\theta$ is closed, by the Poincare lemma, it is locally exact, hence locally, we can always write $\theta = -df$ for some locally defined functin $f$, then $e^f\omega$ is a K\"ahler metric since $d(e^f\omega) = e^f df\wedge\omega + e^f d\omega = e^f(\theta+df)\wedge\omega=0$. Globally, any LCK manifold admits a K\"ahler covering. Indeed, on any cover $\pi:\tilde{M}\to M$ such that $H^1(\tilde{M}, \R) = 0$, the Lee form on the cover $\pi^{\star}\theta$ is globally exact, and the same argument says that $\pi^{\star}\omega$ is globally conformal to a K\"ahler metric. 
\end{Remark}

From this discussion, we see that a locally conformally K\"ahler metric can locally be described by two functions, a K\"ahler potential for the K\"ahler metric $\varphi$ and a conformal factor $e^f$. On a Vaisman manifolds, if $\theta$ is exact and $\theta = -df$, then in \cite{Ver} Verbitsky showed that the metric $\omega$ can be expressed by $ e^{-f}i\d\db e^{f}$. Hence up to a cover, a Vaisman metric can be described by the data of a single potential function $\varphi$ and the metric is given by $\omega = \varphi^{-1}i\d\db\varphi$. Motivated by this, Ornea and Verbitsky introduced a more general class of metrics called LCK metrics with potential in \cite{OV: LCK+P}. 

\begin{Definition}\label{def: LCK metric with potential}
	A locally conformally K\"ahler manifold $(M, J, g)$ has a potential if it has a K\"ahler covering $\pi: \tilde{M}\to M$ with a global positive K\"ahler potential $\varphi\in C^{\infty}_{>0}(\tilde{M})$ such that $\varphi^{-1}i\d\db\varphi = c\pi^{\star}\omega$ for some constant $c$. 
\end{Definition}

It turns out that on this class of manifolds, the $Q$-tensor enjoys an extra symmetry that the Chern curvature does not have. 

\begin{Proposition}\label{prop: LCK with potential are Q-symmetric}
	Suppose a Hermitian manifold $(M, J, g)$ is an LCK manifold with potential, then $Q$ has the additional symmetry $Q_{i\bj k\bl} = Q_{k\bl i \bj}$. 
\end{Proposition}

To prove this, first we need to compute the change of the $Q$ tensor under a conformal change
\begin{Lemma}\label{lem: conformal change}
	If $g$ is a K\"ahler metric and $\tilde{g} = e^{f}g$, then we have
	\begin{equation}
	\tilde{Q}_{i\bj k\bl} = e^{f}\left[R_{i\bj k\bl}-g_{k\bl}f_{i\bj}-f_kf_{\bl}g_{i\bj}-|\d f|^2g_{k\bj}g_{i\bl}+f_kf_{\bj}g_{i\bl}+f_if_{\bl}g_{k\bj}\right]
	\end{equation}
\end{Lemma}
\begin{proof}
	By straightforward computation, we have
	\[\tilde{\Gamma}_{ij}^k =  \tilde{g}^{k\bl}\d_i\tilde{g}_{j\bl} = e^{f}g^{k\bl}\d_i(e^{f}g_{j\bl}) = \Gamma_{ij}^k+f_i\delta_{j}^k\]
	since we assumed $g$ is K\"ahler, we have $\Gamma_{ij}^k = \Gamma_{ji}^k$, so
	\[\tilde{T}_{ij}^k =f_i\delta_{j}^k-f_j\delta_{i}^k\implies \tilde{T}_{ij\bk} = e^{f}(f_ig_{j\bk}-f_jg_{i\bk}) \]
	and 
	\[\tilde{R}_{i\bj k}^{\;\;\;\;m} = -\d_{\bj}\tilde{\Gamma}_{ik}^m = -\d_{\bj}(\Gamma_{ik}^m+f_i\delta_{k}^m) = {R}_{i\bj k}^{\;\;\;\;m} -f_{i\bj}\delta_{k}^m\]
	loweing indices gives
	\[\tilde{R}_{i\bj k\bl} = e^{f}({R}_{i\bj k\bl}-f_{i\bj}g_{k\bl})\]
	substituting into the formula $\tilde{Q}_{i\bj k\bl} = \tilde{R}_{i\bj k\bl}-\tilde{g}^{p\bq}\tilde{T}_{kp\bj}\tilde{\overline{T}}_{lq\bi}$ the result. 
\end{proof}

\begin{proof}[proof of Proposition~\ref{prop: LCK with potential are Q-symmetric}]
	Suppose $g$ satisfies $g = c\varphi^{-1} d d^{c}\varphi$, then by Lemma~\ref{lem: conformal change} with $g = i\d\db \varphi$ and $f = -\log \varphi+\log c$, we get that $Q$ is given by
	\begin{equation}\label{eq: Q under conformal change}
	Q_{i\bj k\bl} = c\varphi^{-1}\left[R^{i\d\db\varphi}_{i\bj k\bl}+\frac{\varphi_{k\bl}\varphi_{i\bj}}{\varphi}-\frac{\varphi_{k\bl}\varphi_i \varphi_{\bj}}{\varphi^2}-\frac{\varphi_{i\bj}\varphi_k\varphi_{\bl}}{\varphi^2}-|\d \log \varphi|^2_{i\d\db\varphi}\varphi_{k\bj}\varphi_{i\bl}+\frac{\varphi_{i\bl}\varphi_k\varphi_{\bj}}{\varphi^2}+\frac{\varphi_{k\bj}\varphi_i\varphi_{\bl}}{\varphi^2}\right]
	\end{equation}
	from which we can read off the symmetry $Q_{i\bj k\bl} = Q_{k\bl i\bj}$. 
\end{proof}

All Vaiman manifolds are LCK with potential, and the Vaisman manifolds are charaterized by the following condition. 
\begin{Theorem}[\cite{OV: Hopf}]\label{thm: Ornea-Verbisky}
	A compact LCK manifold with potential $(M, J, g)$ is Vaisman if and only if $|\theta| = const$. 
\end{Theorem}

\begin{Remark}
	Let $(M, J, g)$ be LCK with potential, then if we pull-back the metric $g$ to its K\"ahler cover and write $\pi^{\star}\omega = c\frac{i\d\db\varphi}{\varphi}$ for some potential $\varphi$, then the condition $|\theta| = const$ is equivalent to $|\d\log \varphi|_{\varphi^{-1}i\d\db\varphi}^2 = 1$. 
\end{Remark}

\begin{Theorem}\label{thm: Vaisman manifold and Q-positivity}
	A Vaisman manifold $(M, J, g)$ is Q-nonnegative if the corresponding K\"ahler metric $\tilde{\omega}$ on its K\"ahler cyclic cover has nonnegative quadratic orthogonal bisectional curvature.
\end{Theorem}
\begin{proof}
	Suppose $M$ is a Vaisman manifold and $\tilde{M}\to M$ be a K\"ahler cyclic cover with nonnegative quadratic orthogonal bisectional curvature. Since Vaiman manifolds are LCK with potential, by Proposition~\ref{prop: LCK with potential are Q-symmetric}, we know that on a Vaisman manifiold $Q$ satisfy the symmetry $Q_{i\bj k\bl} = Q_{k\bl i \bj}$. Moreover, in normal coordinates for $i\d\db \varphi$,  i.e where  $\varphi_{i\bj}=\delta_{i\bj}$, and for $m\neq k$, equation~\ref{eq: Q under conformal change} reduces to
	\begin{align}
	Q_{m\bm k\bk} &= c\varphi^{-1}\left[R^{i\d\db\varphi}_{m\bm k\bk}+\frac{1}{\varphi}-\frac{\varphi_k\varphi_{\bk}+\varphi_m\varphi_{\bm}}{\varphi^2}\right]\\
	& \geq  c\varphi^{-1}\left[R^{i\d\db\varphi}_{m\bm k\bk}+\frac{1-|\d\log \varphi|_{\varphi^{-1}i\d\db\varphi}^2}{\varphi}\right]
	\end{align}
	and by the Vaisman condition, we have $|\d\log \varphi|_{\varphi^{-1}i\d\db\varphi}^2 = 1$. So in the those coordinates, for any $\lambda_1, \ldots, \lambda_n$, we get
	\[\sum_{m, k}Q_{m\bm k\bk}(\lambda_m-\lambda_k)^2 \geq c\varphi^{-1}\sum_{m, k}R^{i\d\db\varphi}_{m\bm k\bk}(\lambda_m-\lambda_k)^2 \geq 0\] 
\end{proof}

\begin{Example}
	A diagonal Hopf surface is one which can be written as $M_{\alpha, \beta} = \C^2\setminus (0, 0)/\sim$ where $(z_1, z_2)\sim (\alpha z_1, \beta z_2)$ for $|\alpha| = |\beta|<1$. These manifold are diffeomorphic to $S^3\times S^1$ and are non-K\"ahler as $b_1 = 1$. They also admit a Vaisman metric given explicitly by
	\begin{equation}
		\omega_{M_{\alpha, \beta}} = \frac{4\sqrt{-1}\d\db |z|^2}{|z|^2}
	\end{equation}
	this metric is $Q$-nonegative by Proposition~\ref{thm: Vaisman manifold and Q-positivity}. In fact this metric is pluriclosed and $Q$ is identically $Q$. One can check that $h^{1, 1}_{BC} = 1$. Thus Theorem~\ref{Main Theorem} implies there exist a Bismut parallel $(1, 1)$-form $\rho$, the two eigenspaces of this form are then Bismut parallel subspaces of $T^{1, 0}$, hence this gives a splitting of the holomorphic tangent bundle with respect to the Bismut connection.  We remark that that Gauduchon and Ornea constructed Vaisman metrics on all class 1 Hopf  surfaces in \cite{OrneaGau}, however the Vaiman metrics on the non-diagonal Hopf surfaces are not $Q$-nonnegative. It would be interesting to know if the non-diagonal Hopf surfaces admit other metrics that are $Q$-nonnegative. 
\end{Example}

\begin{Example}
	A above construction of a diagonal Hopf surface can be generalized to higher dimensions. Define $M_{\alpha} = \C^n\setminus (0, \ldots, 0)/\sim$ where $\alpha = (\alpha_1, \ldots, \alpha_n)$ satisfy $|\alpha_1| = \cdots =|\alpha_n|<1$ and $(z_1, \ldots, z_n)\sim (\alpha_1 z_1, \ldots, \alpha_nz_n)$, then 
	\begin{equation}
		\omega_{M_{\alpha}} = \frac{4\sqrt{-1}\d\db |z|^2}{|z|^2}
	\end{equation}
	is still $Q$-nonnegative by Proposition~\ref{thm: Vaisman manifold and Q-positivity}. In the higher dimensional case, the metrics are no longer pluriclosed and $Q$ does not vanish identically. 
\end{Example}

\paragraph{Acknowledgements:}I would like to thank my advisor Duong Phong for many helpful suggestions and for his constant support and encouragement. I am also grateful to Nikita Klemyatin for bringing to my attention the paper \cite{OV: LCK-Rank}.

\Addresses
\end{document}